\documentclass[10pt]{amsart}
\usepackage{amssymb}

\usepackage{graphicx}
\usepackage{xcolor}

\theoremstyle{plain}
\newtheorem{theorem}{Theorem}[section]
\newtheorem{corollary}[theorem]{Corollary}
\newtheorem{lemma}[theorem]{Lemma}

\newtheorem{proposition}[theorem]{Proposition}

\theoremstyle{remark}

\newtheorem*{remark}{Remarks}

\newtheorem*{claim}{Claim}

\begin{document}

\title[Topologies, descriptive complexity, and equivalence relations]{Transfinite sequences of topologies,\\ descriptive complexity, and\\ approximating equivalence relations}

\author{S{\l}awomir Solecki}
\address{Department of Mathematics, Malott Hall, Cornell University, Ithaca, NY 14853}
\email{ssolecki@cornell.edu}

\begin{abstract} We introduce the notion of filtration between topologies and study its stabilization properties. Descriptive set theoretic 
complexity plays a role in this study. Filtrations lead to natural transfinite sequences approximating a given equivalence relation. We investigate those. 
\end{abstract}

\thanks{Research supported by NSF grant DMS-1800680} 




\maketitle

\section{Introduction}

The aim of the present paper is to describe the following general phenomenon: under appropriate topological conditions, 
increasing transfinite sequences of topologies interpolating between 
two given topologies $\sigma\subseteq \tau$ stabilize at $\tau$ and, under appropriate additional descriptive set theoretic conditions, the stabilization 
occurs at a countable stage of the interpolation. Increasing sequences of topologies play an important role in certain descriptive set theoretic considerations; see, for example, 
\cite[Section 1]{Lo}, \cite[Sections 5.1--5.2]{BK}, \cite[Section 2]{Be}, 
\cite[Section 2]{So1}, \cite[Chapter 6]{Hj1}, \cite[Section 3]{FS}, \cite[Sections 2--4]{So2}, \cite{Hj2}, \cite{Dr}, and, implicitly, 
\cite[Sections 3--5]{BDNT}. 
In this context, such sequences of topologies are often used to approximate an equivalence relation 
by coarser, but more manageable, ones. We relate our theorems on increasing interpolations between two topologies to this theme. Section~\ref{Su:res} contains
a more detailed summary of our results. 
The results of this paper are expected to have applications to a Scott-like analysis of quite general Borel equivalence relations but, since 
they concern a self-contained and, in a way, distinct topic, we decided to publish them separately.

\subsection{Basic notions and notation}

{\em Unless otherwise stated, all topologies are assumed to be defined on a fixed set $X$.} 

We write 
\[
{\rm cl}_\tau\;\hbox{ and }\;{\rm int}_\tau
\]
for the operations of closure and interior with respect to a topology $\tau$. 
If $\tau$ is a topology and $x\in X$, by a {\bf neighborhood of $x$} we understand a subset of $X$ that contains $x$ in its $\tau$-interior. 
A {\bf neighborhood basis of $\tau$} is a family $\mathcal A$ of subsets of $X$ such that for each $x\in X$ and each neighborhood $B$ of $x$, there 
exists $A\in {\mathcal A}$ that is a neighborhood of $x$ and $A\subseteq B$. So a neighborhood basis need not consist of open sets. 
A topology is called {\bf Baire} if a countable union of nowhere dense sets has dense complement.

Given a family of topologies $T$, we write 
\[
\bigvee T
\]
for the topology whose 
basis consist of sets of the form $U_0\cap \cdots \cap U_n$, where each $U_i$, $i\leq n$, is $\tau$-open for some $\tau\in T$. This is the smallest 
topology containing each topology in $T$. If $\tau_i$, for $i\in I$, are topologies, we write 
\[
\bigvee_{i\in I} \tau_i
\]
for $\bigvee T$, where $T=\{ \tau_i\mid  i\in I\}$. 

It is convenient to have the following piece of notation. For an ordinal $\alpha$, let 
\begin{equation}\label{E:oplu} 
\alpha\oplus 1 = 
\begin{cases} 
\alpha+1,&\text{ if $\alpha$ is a successor ordinal;}\\
\alpha, &\text{ if $\alpha$ is equal to $0$ or is a limit ordinal.}
\end{cases}
\end{equation} 
More uniformly, one can write, for all ordinals $\alpha$, 
\[
\alpha\oplus 1 = \sup \{ \xi+2\mid \xi<\alpha\}. 
\]

\subsection{Filtrations}
The notion of filtration defined below is the main new notion of the paper. 
Let $\sigma\subseteq \tau$ be topologies and let $\rho$ be an ordinal. A transfinite sequence $(\tau_\xi)_{\xi<\rho}$ 
of topologies is called a {\bf filtration from $\sigma$ to $\tau$} if
\begin{equation}\label{E:cot}
\sigma= \tau_0\subseteq \tau_1\subseteq \cdots \subseteq \tau_\xi \subseteq \cdots \subseteq \tau
\end{equation}
and, for each $\alpha<\rho$, if $F$ is $\tau_\xi$-closed for some $\xi<\alpha$, then 
\begin{equation}\label{E:intap2}
{\rm int}_{\tau_\alpha}(F) = {\rm int}_{\tau}(F). 
\end{equation}

We will write $(\tau_\xi)_{\xi\leq\rho}$ for $(\tau_\xi)_{\xi<\rho+1}$. 

Each filtration from $\sigma$ to $\tau$ as above can be extended to all ordinals by 
setting $\tau_\xi=\tau$ for all $\xi\geq \rho$. For this reason, it will be harmless to assume that a filtration 
is defined on all ordinals, which we sometimes do to make our notation lighter. On the other hand, a truncation of a filtration 
from $\sigma$ to $\tau$ is also a filtration from $\sigma$ to $\tau$, that is, if $(\tau_\xi)_{\xi<\rho}$ is such a filtration and $\rho'\leq \rho$, 
then so is $(\tau_\xi)_{\xi<\rho'}$. 

A filtration $(\tau_\xi)_{\xi<\rho}$  from $\sigma$ to $\tau$ is also a filtration from $\sigma$ to $\bigvee_{\xi<\rho}\tau_\xi$. 
In fact, if $\tau$ is not relevant to the consideration at hand, 
we call a transfinite sequence $(\tau_\xi)_{\xi<\rho}$ of topologies a {\bf filtration from $\sigma$} if it is a filtration from 
$\sigma$ to $\bigvee_{\xi<\rho}\tau_\xi$. It is easy to see that $(\tau_\xi)_{\xi<\rho}$ is a filtration from $\sigma$ precisely when, 
for each $\alpha<\rho$, $(\tau_\xi)_{\xi\leq\alpha}$ is a filtration from $\sigma$ to $\tau_\alpha$.

Note that if $F\subseteq X$ is an arbitrary set and $(\tau_\xi)_\xi$ is a transfinite sequence of topologies fulfilling \eqref{E:cot}, then for each $\alpha$
\[
{\rm int}_{\tau_\alpha}(F)\subseteq {\rm int}_\tau(F).
\] 
So condition \eqref{E:intap2} 
says that if $F$ is simple from the point of view of $\tau_\alpha$, that is, if $F$ is $\tau_\xi$-closed for some 
$\xi< \alpha$, then ${\rm int}_{\tau_\alpha}(F)$ is as large as possible, in fact, equal to ${\rm int}_\tau(F)$. 
One might say that if $F$ is $\tau_\xi$-closed for some $\xi< \alpha$, then $\tau_\alpha$ computes the interior of $F$ correctly, that is, as intended by $\tau$. 
In some results below, we will find it useful to consider a weakening of  \eqref{E:intap2} to \eqref{E:intap}.

\subsection{Results}\label{Su:res}
Let $\sigma\subseteq \tau$ be two topologies. 
The first question is to determine whether a given filtration $(\tau_\xi)_\xi$ from $\sigma$ to $\tau$ reaches $\tau$, that is, whether there exists an ordinal 
$\xi$ with $\tau_\xi=\tau$. 
Since all the topologies $\tau_\xi$ are defined on the same set, there exists an ordinal $\xi_0$ such that $\tau_\xi= \tau_{\xi_0}$ for all $\xi\geq \xi_0$; 
the question is whether $\tau_{\xi_0}=\tau$. 
If the answer happens to be positive, we aim to obtain information on 
the smallest ordinal $\xi$ for which $\tau_\xi=\tau$. We will achieve these goals in Sections~\ref{S:stbd} and \ref{S:stdes} 
(Corollary~\ref{C:tst}, Theorem~\ref{T:stab2}, and Corollary~\ref{C:stom}) 
assuming that $\tau$ is regular and Baire and that it has a neighborhood basis consisting 
of sets that are appropriately definable with respect to $\sigma$. So, informally speaking, termination at $\tau$ of a filtration from $\sigma$ to $\tau$ 
has to do with the attraction exerted by $\tau$, 
which is expressed by $\tau$ being Baire, and with the distance from $\sigma$ to $\tau$, which is expressed by the complexity, with 
respect to $\sigma$, of a neighborhood basis of $\tau$. 

Given an equivalence relation $E$ on a set $X$, with $X$ equipped with a topology $\tau$, we can define a canonical equivalence relation 
that approximates $E$ from above: make $x,y\in X$ equivalent 
when the $\tau$-closures of the $E$ equivalence classes of $x$ and $y$ are equal. Given a filtration, this procedure gives rise 
to a transfinite sequence of upper approximations of $E$. In Section~\ref{S:eqr}, we consider the question 
of these approximations stabilizing to $E$. We answer it in Theorem~\ref{T:eqte} and 
Corollary~\ref{C:ceq}. 

We also present and study a canonical, slowest filtration from $\sigma$ to $\tau$; see Section~\ref{S:slf}.

\section{The slowest filtration}\label{S:slf}

We introduce an operation on pairs of topologies, which will let us define filtrations. Let $\sigma$ and $\tau$ be topologies. Let 
\begin{equation}\label{E:sit}
(\sigma,\tau)
\end{equation}
be the family of all unions of sets of the form 
\[
U\cap {\rm int}_\tau(F), 
\]
where $U$ is $\sigma$-open and $F$ is $\sigma$-closed. Since 
\[
{\rm int}_\tau(F_1\cap F_2) = {\rm int}_\tau(F_1)\cap {\rm int}_\tau(F_2),
\]
it follows that $(\sigma, \tau)$ is a topology.

We record the following obvious lemma. 

\begin{lemma}\label{L:opb} 
Let $\sigma\subseteq \tau$ be topologies. 
\begin{enumerate}
\item[(i)] We have $\sigma\subseteq (\sigma,\tau)\subseteq \tau$.

\item[(ii)] If $(\tau_\xi)_\xi$ be a filtration from $\sigma$ to $\tau$, then $\tau_\xi\subseteq (\tau_\xi,\tau)\subseteq \tau_{\xi+1}$, for each $\xi$.
\end{enumerate}
\end{lemma}

Let $\sigma$ and $\tau$ be two topologies with $\sigma\subseteq\tau$. Lemma~\ref{L:opb} suggests defining a filtration from $\sigma$ to $\tau$ that would be 
the slowest such filtration; see Proposition~\ref{P:slo} below.
This goal will be achieved by extending operation \eqref{E:sit} to a transfinite sequence of topologies. 
So we define by transfinite recursion topologies $(\sigma,\tau)_\xi$, where $\xi$ is an ordinal. (We will have $(\sigma,\tau)_1=(\sigma,\tau)$.) 
Let 
\[
(\sigma, \tau)_0=\sigma. 
\]
If $(\sigma,\tau)_\xi$ has been defined, let 
\[
(\sigma, \tau)_{\xi+1} = ((\sigma,\tau)_\xi, \tau).
\]
If $\lambda$ is a limit ordinal and $(\sigma,\tau)_\xi$ have been defined for all $\xi<\lambda$, then 
\[
(\sigma, \tau)_\lambda=\bigvee_{\xi<\lambda}(\sigma, \tau)_\xi.
\] 

Note that the definition above can be phrased as follows. Given an ordinal $\xi$, 
if $(\sigma, \tau)_\gamma$ are defined for all $\gamma<\xi$, then $(\sigma, \tau)_\xi$ is the family of all unions of sets of the form 
\[
U\cap {\rm int}_\tau(F)
\]
where, for some $\gamma<\xi$,  $U$ is $(\sigma, \tau)_\gamma$-open and $F$ is $(\sigma, \tau)_\gamma$-closed.

Proposition~\ref{P:slo} justifies regarding $((\sigma,\tau)_\xi)_\xi$ as the slowest filtration from $\sigma$ to $\tau$. On the opposite end, 
the transfinite sequence $(\tau_\xi)_\xi$ 
with $\tau_0=\sigma$ and $\tau_\xi = \tau$ for $\xi>0$ is trivially the fastest such filtration. 

\begin{proposition}\label{P:slo} 
Let $\sigma\subseteq \tau$ be topologies. 
\begin{enumerate}
\item[(i)] The transfinite sequence $((\sigma,\tau)_\xi)_\xi$ is a filtration from $\sigma$ to $\tau$. 

\item[(ii)] If $(\tau_\xi)_\xi$ is a filtration from $\sigma$ to $\tau$, then $(\sigma,\tau)_\xi\subseteq \tau_\xi$, for each ordinal $\xi$. 
\end{enumerate}
\end{proposition}

\begin{proof} Immediately from Lemma~\ref{L:opb}(i), we get 
\[
\sigma= (\sigma,\tau)_0\subseteq (\sigma,\tau)_1\subseteq\cdots  \subseteq (\sigma,\tau)_{\xi}\subseteq\cdots \subseteq \tau.
\]
It is also clear from the very definition that, for each $\alpha$, if $F$ is $(\sigma,\tau)_\xi$-closed for some $\xi<\alpha$, then 
\[
{\rm int}_{(\sigma,\tau)_\alpha}(F) = {\rm int}_\tau(F),
\]
that is, we have point (i).  

Point (ii) is obtained by transfinite induction. Clearly, we have $(\sigma,\tau)_0 = \sigma= \tau_0$. Assuming inductively that 
$(\sigma,\tau)_\xi\subseteq \tau_\xi$ and using Lemma~\ref{L:opb}(ii), we get 
\[
(\sigma, \tau)_{\xi+1} = ((\sigma,\tau)_\xi, \tau)\subseteq (\tau_\xi, \tau)\subseteq \tau_{\xi+1},
\]
as required. If $\lambda$ is a limit ordinal and  if, inductively, $(\sigma,\tau)_\xi\subseteq \tau_\xi$ 
for all $\xi<\lambda$, then $\bigcup_{\xi<\lambda}(\sigma, \tau)_\xi \subseteq \tau_\lambda$ and, therefore, $(\sigma, \tau)_\lambda\subseteq \tau_\lambda$. 
The conclusion follows. 
\end{proof}

\section{Stabilization at $\tau$}\label{S:stbd}

Theorem~\ref{T:sts} should be seen in the context of Lemma~\ref{L:opb}(i). 

\begin{theorem}\label{T:sts}
Let $\sigma\subseteq \tau$ be topologies. 
Assume that $\tau$ is regular, Baire, and has a neighborhood basis consisting of sets with the Baire property with respect to $\sigma$. 
If  $\sigma=(\sigma,\tau)$, then $\sigma= \tau$. 
\end{theorem}

We start with a general lemma that will be used here and later on to check equality of two topologies.

\begin{lemma}\label{L:toe}
Let $Z$ be a regular topological space, and let $Y$ be a Baire space. 
Let $f\colon Z\to Y$ be a continuous bijection. Assume that, for each $z\in Z$ and a non-empty open $z\in U\subseteq Z$, 
$f(U)$ is comeager in a neighborhood of $f(z)$. Then $f$ is a homeomorphism.  
\end{lemma}

\begin{proof} We write ${\rm cl}_Z$ for closure in $Z$. 

We show that, for each $z\in U\subseteq Z$, with $U$ open, 
$f({\rm cl}_Z(U))$ contains $f(z)$ in its interior. If not, then, by surjectivity of $f$, 
$f(Z\setminus {\rm cl}_Z(U))$ has $f(z)$ in its closure. Since $Z\setminus {\rm cl}_Z(U)$ is open, we have that $f(Z\setminus {\rm cl}_Z(U))$ is non-meager in each 
neighborhood of each of its points. Since each neighborhood of $z$ contains a point in $f(Z\setminus {\rm cl}_Z(U))$, it follows that $f(Z\setminus {\rm cl}_Z(U))$ 
is non-meager in each neighborhood of $f(z)$. By injectivity of $f$ and $Y$ being Baire, this statement contradicts $f(U)$ being comeager in a neighborhood of $f(z)$.

Now we finish the proof by noticing that, by regularity of $Z$, for each $U\subseteq Z$ open we have 
\[
U = \bigcup_{z\in U} {\rm cl}_Z(U_z)
\]
for some open sets $U_x$ with $z\in U_z$. Thus, 
\[
f(U) = \bigcup_{z\in U} f({\rm cl}_Z(U_z))
\]
and, by what was proved above, $f({\rm cl}_Z(U_z))$ contains $f(z)$ in its interior. Thus, $f(U)$ is open, and the lemma follows. 
\end{proof}

\begin{proof}[Proof of Theorem~\ref{T:sts}]
First, we claim that each non-empty $\tau$-open set is non-meager with respect to $\sigma$. 
Let $V$ be non-empty and $\tau$-open, and, towards a contradiction, assume that we have closed and nowhere dense sets with respect to $\sigma$ sets 
$F_n$, $n\in {\mathbb N}$, such that $\bigcup_nF_n\supseteq V$.
Then ${\rm int}_\tau(\bigcup_n F_n)\not= \emptyset$. Since $\tau$ is Baire and each $F_n$ is also $\tau$-closed, 
it follows that ${\rm int}_\tau(F_{n_0})\not= \emptyset$, for some $n_0$. 
Since $F_{n_0}$ is $\sigma$-closed, we have that ${\rm int}_\tau(F_{n_0})$ 
is $(\sigma,\tau)$-open, so 
since $(\sigma,\tau) =\sigma$, it is $\sigma$-open. Thus, ${\rm int}_\sigma(F_{n_0})\not= \emptyset$ contradicting the assumption on 
the sequence $(F_n)$. 

Our second claim is that for each $x\in X$, each $\tau$-neighborhood of $x$ is $\sigma$-dense in a $\sigma$-neighborhood of $x$. 
Indeed, let $V$ be a $\tau$-open set containing $x$. Then ${\rm cl}_\sigma(V)$ is $\sigma$-closed and, therefore, ${\rm int}_\tau({\rm cl}_\sigma(V))$ is 
$(\sigma,\tau)$-open and so $\sigma$-open since $(\sigma,\tau)= \sigma$. We clearly have 
\[
x\in V\subseteq {\rm int}_\tau({\rm cl}_\sigma(V))
\]
and $V$ is $\sigma$-dense in ${\rm int}_\tau({\rm cl}_\sigma(V))$. It follows that ${\rm int}_\tau({\rm cl}_\sigma(V))$ is a $\sigma$-neighborhood of $x$, in which 
$V$ is $\sigma$-dense. 

Thirdly, we observe that, by assumption, each $x\in X$ has a $\tau$-neighborhood basis consisting of sets that have the Baire property with respect to $\sigma$. 

It follows immediately, from the three claims above, that for each $x\in X$, each $\tau$-neighborhood of $x$ is $\sigma$-comeager in a $\sigma$-neighborhood 
of $x$. The first claim also implies that the topology $\sigma$ is Baire.

The above observation implies the conclusion of the theorem by Lemma~\ref{L:toe} applied to ${\rm id}_X\colon (X, \tau)\to (X, \sigma)$. 
\end{proof}

If $(\tau_\xi)_\xi$ is a filtration from $\sigma$ to $\tau$, an intuition behind condition \eqref{E:intap2} is that it 
tries to ensure that $\tau_{\xi+1}$ is substantially closer to $\tau$ than $\tau_\xi$, unless $\tau_\xi$ is already equal to $\tau$. 
Corollary~\ref{C:tst}(ii) below resonates with this intuition. Proposition~\ref{P:slo}(ii) suggests regarding the smallest $\xi$ as in 
the conclusion of Corollary~\ref{C:tst}(ii) as an ordinal valued ``distance"  from $\sigma$ to $\tau$. 

Recall that {\bf C-sets} with respect to a topology is the smallest $\sigma$-algebra of sets closed under the Souslin operation and containing all 
open sets with respect to this topology; see \cite[Section~29D]{Ke}. 
The main point for us is that C-sets have the Baire property even if the given topology is strengthened; see \cite[Corollary~29.14]{Ke}. 

\begin{corollary}\label{C:tst}
Let $\sigma\subseteq \tau$ be topologies. 
Assume that $\tau$ is regular, Baire, and has a neighborhood basis consisting of sets that are C-sets with respect to $\sigma$. 
\begin{enumerate}
\item[(i)] Let $(\tau_\xi)_\xi$ be a filtration from $\sigma$ to $\tau$. If $\tau_{\xi_0} =\tau_{\xi_0+1}$, then $\tau_{\xi_0}= \tau$. 

\item[(ii)] There exists an ordinal $\xi$ such that $(\sigma,\tau)_\xi=\tau$. 
\end{enumerate}
\end{corollary}

\begin{proof} (i) Let $\xi$ be such that $\tau_\xi=\tau_{\xi+1}$. This equality and Lemma~\ref{L:opb}(ii) give $\tau_\xi=(\tau_\xi, \tau)$. 
Now the conclusion follows from Theorem~\ref{T:sts} if we only notice that C-sets with respect to $\sigma$ are also C-sets 
with respect to $\tau_\xi$ since $\sigma=\tau_0\subseteq\tau_\xi$ and, therefore, they have the Baire property with respect to $\tau_\xi$. 

(ii) Since the topologies $(\sigma,\tau)_\xi$ are defined on the same set $X$ for all ordinals $\xi$, there exists an ordinal $\xi$ such that 
$(\sigma,\tau)_\xi = (\sigma,\tau)_{\xi+1}$, and (ii) follows from (i). 
\end{proof}

\section{Stabilization at $\tau$ and descriptive set theoretic complexity}\label{S:stdes}

We prove here a more refined version of stabilization. Theorem~\ref{T:stab2} makes a connection with descriptive set theoretic complexity of neighborhood bases. 
Note that the assumptions of Theorem~\ref{T:stab2} ensure that Corollary~\ref{C:tst}(i) applies, 
but the conclusion of Theorem~\ref{T:stab2} gives an upper estimate on the smallest $\xi_0$ with $\tau_{\xi_0}=\tau$, which we do not get from Corollary~\ref{C:tst}(i).

\begin{theorem}\label{T:stab2}
Let $\sigma\subseteq \tau$ be topologies, 
with $\tau$ being regular and Baire. For an ordinal $\alpha\leq \omega_1$, 
let $(\tau_\xi)_{\xi\leq\alpha}$ be a filtration from $\sigma$ to $\tau$, with $\tau_\xi$ metrizable, for $\xi<\alpha$, and $\tau_\alpha$ Baire. 

If $\tau$ has a neighborhood basis consisting of sets in $\bigcup_{\xi<\alpha}{\mathbf \Pi}^0_{1+\xi}$ with respect to $\sigma$, then 
$\tau_\alpha=\tau$. 
\end{theorem}

\begin{remark} 
{\bf 1.} We emphasize that in Theorem~\ref{T:stab2} we do not make any separability assumptions.

{\bf 2.} One can relax the assumption of metrizability but with no apparent gain in applicability; 
it suffices to assume that $\tau_\xi$ are paracompact and that sets that are $\tau_\xi$-closed are 
intersections of countably many sets that are $\tau_\xi$-open, for all $\xi<\alpha$. 

{\bf 3.} When $\alpha=\omega_1$, then, of course, $\bigcup_{\xi<\alpha}{\mathbf \Pi}^0_{1+\xi}$ 
is the family of all Borel sets with respect to $\sigma$. 
\end{remark}

Fix $(\tau_\xi)_{\xi<\rho}$, a transfinite sequence of topologies fulfilling \eqref{E:cot}. 

Let $\alpha<\rho$. 
Define {\bf $\alpha$-tame} sets to be the smallest family of subsets of $X$ containing $\tau_\xi$-closed sets for each $\xi<\alpha$ 
and closed under the following operation. Let $\mathcal U$ be a $\tau_\xi$-discrete family of $\tau_\xi$-open sets, for some $\xi<\alpha$. 
Let $F^U$ be an $\alpha$-tame set, for $U\in {\mathcal U}$. Then 
\[
\bigcup_{U\in {\mathcal U}} \left(F^U\cap U\right)
\]
is $\alpha$-tame. 

The class of $\alpha$-tame sets is needed in the proof of Theorem~\ref{T:stab2} to handle the case $\alpha=\omega_1$. 
If $\alpha<\omega_1$, the simpler family of sets 
that are $\tau_\xi$-closed for $\xi<\alpha$ suffices. This is reflected in Lemma~\ref{L:fsi}(ii) below.

\begin{lemma}\label{L:fsi} 
Let $(\tau_\xi)_{\xi<\rho}$ be a transfinite sequence fulfilling (1), and let $\alpha<\rho$. 
\begin{enumerate}
\item[(i)] If $\tau_\xi$ is metrizable, for each $\xi<\alpha$, then each $\alpha$-tame set is a countable union of $\tau_\alpha$-closed sets. 

\item[(ii)] If $\alpha<\omega_1$ and $\tau_\xi$ is metrizable, for each $\xi<\alpha$, then each $\alpha$-tame set is a countable union of $\tau_\xi$-closed sets with 
$\xi<\alpha$. That is, for each $\alpha$-tame set $F$, there exist sets $F_n$, $n\in {\mathbb N}$, such that $F_n$ is $\tau_{\xi_n}$-closed, for some $\xi_n<\alpha$, 
and $F= \bigcup_n F_n$. 
\end{enumerate}
\end{lemma}

\begin{proof} 
We prove point (i). It suffices to show that the family of countable unions of $\tau_\alpha$-closed sets is closed under the operation 
in the definition of $\alpha$-tame sets. Let $\mathcal U$ be a $\tau_\xi$-discrete family of $\tau_\xi$-open sets, for some 
$\xi<\alpha$, and let $F_k^U$, for $U\in {\mathcal U}$, $k\in {\mathbb N}$, 
be $\tau_\alpha$-closed sets. We need to see that 
\begin{equation}\label{E:unon}
\bigcup_{U\in {\mathcal U}} \left(\left(\bigcup_k F_k^U\right)\cap U\right)
\end{equation} 
is a countable union of $\tau_\alpha$-closed sets. Recall that each $\tau_\xi$-open set is a countable union of $\tau_\xi$-closed sets. So, for each 
$U\in {\mathcal U}$, we can fix $\tau_\xi$-closed sets $H^U_n$, $n\in {\mathbb N}$, such that $U=\bigcup_n H^U_n$. Thus, the set in \eqref{E:unon} 
can be represented as 
\begin{equation}\label{E:unond}
\bigcup_k\bigcup_n \left( \bigcup_{U\in {\mathcal U}} \left(F^U_k\cap H^U_n\right)\right)
\end{equation}
and $\tau_\xi$-discreteness of $\mathcal U$, which implies $\tau_\alpha$-discreteness of $\mathcal U$, ensures that the sets 
$\bigcup_{U\in {\mathcal U}} \left(F^U_k\cap H^U_n\right)$ are $\tau_\alpha$-closed. 

The proof of (ii) is similar to the proof of (i). With the notation (${\mathcal U},\, \xi,\, F^U_k$) as above, we assume that for each $U\in {\mathcal U}$ and $k\in {\mathbb N}$, 
the set $F_k^U$ is $\tau_{\gamma^U_k}$-closed for some $\gamma^U_k<\alpha$. Working with formula \eqref{E:unond}, we see that 
\begin{equation}\label{E:remi}
\bigcup_{U\in {\mathcal U}} \left(F^U_k\cap H^U_n\right) = \bigcup_{\gamma<\alpha} \bigcup \{ F^U_k\cap H^U_n\mid \gamma^U_k=\gamma\}. 
\end{equation} 
Now, the first union on the right hand side of \eqref{E:remi} is countable and, by $\tau_\xi$-discreteness of $\mathcal U$, we see that 
\[
\bigcup \{ F^U_k\cap H^U_n\mid \gamma^U_k=\gamma\}
\]
is $\tau_{\max(\xi, \gamma)}$-closed. Since $\max(\xi, \gamma)<\alpha$, point (iii) and the lemma follow. 
\end{proof}

We say that $A\subseteq X$ is {\bf $\alpha$-solid} if for each countable family $\mathcal F$ of $\alpha$-tame  sets with 
$\bigcup {\mathcal F}$ containing a non-empty relatively $\tau_\alpha$-open subset of $A$, we have  
${\rm int}_{\tau_\alpha}(F)\not= \emptyset$ for some $F\in {\mathcal F}$. 
We call a set $A\subseteq X$ {\bf $\alpha$-slight} if there exists a countable family $\mathcal F$ with $A\subseteq \bigcup {\mathcal F}$ and such that 
each $F\in {\mathcal F}$ is $\alpha$-tame and ${\rm int}_{\tau_\alpha}(F) =\emptyset$. 
We register the following lemma that follows directly from the definitions. 

\begin{lemma}\label{L:obv} 
Let $(\tau_\xi)_{\xi<\rho}$ be a transfinite sequence fulfilling (1), and let $\alpha<\rho$. 
A set is $\alpha$-solid if and only if no non-empty relatively $\tau_\alpha$-open subset of it is $\alpha$-slight.
\end{lemma}

Lemma~\ref{L:onsl} below that contains basic properties of $\alpha$-slight sets. 

\begin{lemma}\label{L:onsl}
Let $(\tau_\xi)_{\xi<\rho}$ be a transfinite sequence fulfilling (1), and let $\alpha<\rho$. 
\begin{enumerate}
\item[(i)] The empty set is $\alpha$-slight. 

\item[(ii)] If $\tau_\xi$ is metrizable, for each $\xi<\alpha$, then $\alpha$-slight sets are $\tau_\alpha$-meager. 

\item[(iii)] Countable unions of $\alpha$-slight sets are $\alpha$-slight. 

\item[(iv)] Assume $\tau_\xi$ is metrizable, for each $\xi<\alpha$. 
Let $A\subseteq X$. Assume that for some $\xi<\alpha$, there is a family $\mathcal U$ of $\tau_\xi$-open sets such that $A\subseteq \bigcup {\mathcal U}$ and 
$A\cap U$ is $\alpha$-slight for each $U\in {\mathcal U}$. Then $A$ is $\alpha$-slight. 

\item[(v)] Assume that $\alpha<\omega_1$ and $\tau_\xi$ is metrizable, for each $\xi<\alpha$. 
Let $A\subseteq X$. Then $A$ is $\alpha$-slight if and only if there is a countable family $\mathcal F$ such that $A\subseteq \bigcup {\mathcal F}$ and 
each $F\in {\mathcal F}$ is $\tau_\xi$-closed, for some $\xi<\alpha$ depending on $F$, and ${\rm int}_{\tau_\alpha}(F)=\emptyset$. 
\end{enumerate}
\end{lemma}

\begin{proof} Points (i) and (iii) are obvious, with point (i) for $\alpha=0$ being true due to the set theoretic convention that 
the union of an empty family of sets is the empty set. Point (ii) is immediate from Lemma~\ref{L:fsi}(i). 

We show (iv). 
By Stone's Theorem \cite[Theorem~4.4.1]{En}
each family of $\tau_\xi$-open sets has a refinement that is a $\sigma$-discrete, 
with respect to $\tau_\xi$, family of $\tau_\xi$-open sets. 
Therefore, there are $\tau_\xi$-discrete families ${\mathcal U}_n$, $n\in {\mathbb N}$, of $\tau_\xi$-open sets such that, given $n$, 
the set $A\cap U$ is $\alpha$-slight for each $U\in {\mathcal U}_n$ and 
the family $\bigcup_n {\mathcal U}_n$ covers $A$. Thus,
by (iii), it suffices to show the following statement: if $\mathcal U$ is a $\tau_\xi$-discrete family of $\tau_\xi$-open sets 
covering $A$ such that $A\cap U$ is $\alpha$-slight, for each $U\in {\mathcal U}$, then $A$ is $\alpha$-slight. 
Now, for $U\in {\mathcal U}$, we can find a sequence $F_k^U$, $k\in {\mathbb N}$, of $\alpha$-tame sets 
such that $A\cap U \subseteq \bigcup_k F^U_k$ and ${\rm int}_{\tau_\alpha}( F^U_k)=\emptyset$ for each $k$. 
By $\tau_\xi$-discreteness of ${\mathcal U}$, for each $k\in {\mathbb N}$, the set 
\[
E_k= \bigcup \{ F^U_k\cap U \mid U\in {\mathcal U} \}
\]
is $\alpha$-tame. Again, using $\tau_\xi$-discreteness of $\mathcal U$, we see that,  for each $k$, $E_k$ 
has empty interior with respect to $\tau_\alpha$. Since $A \subseteq \bigcup_k E_k$, we see that 
$A$ is $\alpha$-slight. 

To see (v), note first that the direction $\Leftarrow$ is obvious since for each $\xi<\alpha$, each $\tau_\xi$-closed set is $\alpha$-tame. The 
direction $\Rightarrow$ follows from Lemma~\ref{L:fsi}(ii). 
\end{proof}

We record a reformulation of Lemma~\ref{L:onsl}(iv) that we will need later in the paper.

\begin{lemma}\label{L:onso}
Let $(\tau_\xi)_{\xi<\rho}$ be a transfinite sequence fulfilling (1), and let $\alpha<\rho$. 
Assume that $\alpha<\omega_1$ and $\tau_\xi$ is metrizable, for each $\xi<\alpha$. 
Then $A\subseteq X$ is $\alpha$-solid if and only if
for each countable family $\mathcal F$ of sets such that each $F\in {\mathcal F}$ is $\tau_\xi$-closed, for some $\xi<\alpha$ depending on $F$, and 
$\bigcup {\mathcal F}$ contains a non-empty relatively $\tau_\alpha$-open subset of $A$, we have  
${\rm int}_{\tau_\alpha}(F)\not= \emptyset$ for some $F\in {\mathcal F}$. 
\end{lemma}

\begin{proof} The lemma follows from Lemmas~\ref{L:onsl}(iv) and \ref{L:obv}. 
\end{proof}

We say that $(\tau_\xi)_{\xi<\rho}$ is a {\bf weak filtration from $\sigma$ to $\tau$} provided (1) holds and, for each $\alpha<\rho$, 
if $F$ is $\tau_\xi$-closed for some $\xi<\alpha$, then 
\begin{equation}\label{E:intap}
{\rm int}_{\tau_\alpha}(F) \hbox{ is $\tau$-dense in }  {\rm int}_{\tau}(F). 
\end{equation}
It is clear that each filtration is a weak filtration. 

\begin{lemma}\label{L:frth}
Let $(\tau_\xi)_{\xi<\rho}$ be a weak filtration from $\sigma$ to $\tau$. 
\begin{enumerate}
\item[(i)] If $\alpha<\rho$ and $F$ is $\alpha$-tame, then 
${\rm int}_{\tau_\alpha}(F) \hbox{ is $\tau$-dense in }  {\rm int}_{\tau}(F)$. 

\item[(ii)] If $\alpha\leq\beta<\rho$, then $\beta$-solid sets are $\alpha$-solid. 
\end{enumerate}
\end{lemma} 

\begin{proof} (i) We consider the family 
\[
{\mathcal F} = \{ F\mid F\subseteq X,\,  {\rm int}_{\tau_\alpha}(F)\hbox{ is $\tau$-dense in }{\rm int}_{\tau}(F)\}. 
\]
We need to show that $\alpha$-tame sets are included in $\mathcal F$. Since $(\tau_\xi)_{\xi<\rho}$ is a weak filtration, $\mathcal F$ 
contains all $\tau_\xi$-closed sets for all $\xi<\alpha$. It remains to see that $\mathcal F$ is closed under the operation in the definition 
of $\alpha$-tame sets. 
Let $\mathcal U$ be a $\tau_\xi$-discrete family of $\tau_\xi$-open sets, for some $\xi<\alpha$. Let $F^U$ be sets in $\mathcal F$, for $U\in {\mathcal U}$. 
We need to see that 
\[
\bigcup_{U\in {\mathcal U}} \left(F^U\cap U\right)\in {\mathcal F}. 
\]

Since the family $\mathcal U$ is $\tau_\xi$-discrete, so $\tau$-discrete, and sets in $\mathcal U$ are $\tau_\xi$-open, so $\tau$-open, we have 
\[
{\rm int}_\tau\left( \bigcup_{U\in {\mathcal U}} \left(F^U\cap U\right)\right) = \bigcup_{U\in {\mathcal U}} {\rm int}_\tau\left(F^U\cap U\right)= 
\bigcup_{U\in {\mathcal U}} \left( {\rm int}_\tau\left(F^U\right)\cap U\right). 
\]
Since  each $U\in {\mathcal U}$ is also $\tau_\alpha$-open, we have 
\[
{\rm int}_{\tau_\alpha} \left(F^U\cap U\right)   =    {\rm int}_{\tau_\alpha} \left(F^U\right) \cap U.
\]
It follows that it is enough to show that ${\rm int}_{\tau_\alpha} \left(F^U\right) \cap U$ is $\tau$-dense in ${\rm int}_\tau \left(F^U\right) \cap U$, 
for each for $U\in {\mathcal U}$. But this is clear since, by assumption, ${\rm int}_{\tau_\alpha} \left(F^U\right)$ is $\tau$-dense in ${\rm int}_\tau \left(F^U\right)$ and $U$ is $\tau$-open being $\tau_\xi$-open. 

(ii) We make two observations. First, clearly, $\alpha$-tame sets are $\beta$-tame. 
Second, for $F\subseteq X$, ${\rm int}_{\tau_\beta}(F)\not=\emptyset$ trivially implies that 
${\rm int}_{\tau}(F)\not= \emptyset$. If now $F$ is $\alpha$-tame, then 
${\rm int}_{\tau}(F)\not= \emptyset$ implies ${\rm int}_{\tau_\alpha}(F) \not=\emptyset$, by (i). 
So for $\alpha$-tame sets, ${\rm int}_{\tau_\beta}(F)\not=\emptyset$ implies ${\rm int}_{\tau_\alpha}(F) \not=\emptyset$. 
These two observations give (ii). 
\end{proof}

The statement of the following technical result is more precise than what is needed in this section, but this more refined 
version will be used in Section~\ref{S:eqr}. Its proof extends the arguments in \cite[Lemma~4.1]{So2}. There are also analogies with \cite[Lemmas 8 and 9]{Lo}.

\begin{lemma}\label{L:stab}
Let $\alpha\leq\omega_1$. Assume that $(\tau_\xi)_{\xi\leq\alpha}$ is a weak filtration from $\sigma$, with $\tau_\xi$ metrizable for $\xi<\alpha$. 
If $A\subseteq X$ is ${\mathbf \Pi}^0_{1+\xi}$ with respect to $\sigma$, for some $\xi\leq\alpha$, $\xi<\omega_1$, and $B\subseteq A$ is $\alpha$-solid, 
then ${\rm cl}_{\tau_\xi}(B)\setminus A$ is $\tau_\alpha$-meager. 
\end{lemma}

For the remainder of the proof of Lemma~\ref{L:stab}, we fix $\alpha$ and $(\tau_\xi)_{\xi\leq\alpha}$ as in the statement. 
For $\xi\leq\alpha$, put 
\begin{equation}\label{E:shclin}
{\rm cl}_\xi = {\rm cl}_{\tau_\xi} \;\hbox{ and }\;{\rm int}_\xi = {\rm int}_{\tau_\xi}. 
\end{equation}
Since $\alpha$ is fixed, we write {\bf slight} for $\alpha$-slight.

\begin{lemma}\label{L:slal}
If $A\subseteq X$ is ${\mathbf \Pi}^0_{1+\xi}$ with respect to $\sigma$, for $\xi\leq \alpha$, $\xi<\omega_1$, then there exists a $\tau_\xi$-closed set $F$ such that 
\begin{enumerate}
\item[(i)] if $\xi<\alpha$, then $(A\setminus F)\cup (F\setminus A)$ is slight; 

\item[(ii)] if $\xi=\alpha$, then $F\setminus A$ is $\tau_\alpha$-meager and $A\setminus F$ is the union of $\tau_\alpha$-open sets $U$ such that $U\cap A$ 
is slight.
\end{enumerate}
\end{lemma}

\begin{proof} First we make the following technical observation. 

\noindent {\em Let $\gamma<\xi\leq\alpha$, and let $A, F_1, F_2\subseteq X$ be such that $F_1$ is $\tau_\gamma$-closed, 
$F_2$ is $\tau_\xi$-closed, $A\cap F_1$ is slight, and $A\cap V$ is not slight for each $\tau_\xi$-open set $V$ with $V\cap F_2\not=\emptyset$. Then 
\[
{\rm int}_\alpha(F_1\cap F_2)=\emptyset.
\]}

To prove this observation, 
set $U={\rm int}_\alpha(F_1\cap F_2)$ and assume that $U$ is not empty. 
Note that $U$ is $\tau_\alpha$-open and $U\subseteq F_1$. Since $\gamma<\xi\leq \alpha$ and $F_1$ is 
$\tau_\gamma$-closed, by assumption \eqref{E:intap}, there exists a $\tau_\xi$-open set $V$ with 
\begin{equation}\label{E:uv}
V\subseteq F_1 \;\hbox{ and }\; V\cap U\not=\emptyset.
\end{equation}
Since $A\cap F_1$ is assumed to be slight, by the first part of \eqref{E:uv}, 
$A\cap V$ is slight, which implies by our assumption on $F_2$ that $V\cap F_2=\emptyset$.  
Now, by the second part of \eqref{E:uv}, we get $U\not\subseteq F_2$, which leads to a contradiction with the definition 
of $U$. 

For $A\subseteq X$ and $\xi\leq \alpha$, let 
\[
c_\xi(A) = X\setminus \bigcup \{ U\mid A\cap U\hbox{ is slight and } U \hbox{ is $\tau_\xi$-open}\}.
\] 

We show that $F= c_\xi(A)$ fulfills the conclusion of the lemma. 
Obviously $c_\xi(A)$ is $\tau_\xi$-closed. It is clear that $A\setminus c_\alpha(A)$ fulfills the second part of point (ii). By Lemma~\ref{L:onsl}(iv), 
if $\xi<\alpha$, then $A\setminus c_\xi(A)$ is slight.

It remains to see that if $A$ is ${\mathbf \Pi}^0_{1+\xi}$ with respect to $\tau_0$, 
then $c_\xi(A)\setminus A$ is slight, if $\xi<\alpha$, and $c_\xi(A)\setminus A$ is $\tau_\alpha$-meager, 
if $\xi=\alpha$. 
This is done by induction on $\xi$. For $\xi=0$, $A$ is ${\mathbf \Pi}^0_1$ with respect to $\tau_0$, so  $c_0(A)\subseteq A$ by Lemma~\ref{L:onsl}(i). 
Now $c_0(A)\setminus A$ being empty is $\tau_\alpha$-meager if $\alpha=0$ and is slight if $\alpha>0$ by Lemma~\ref{L:onsl}(i). 
Assume we have the conclusion for all $\gamma<\xi$. Let $A$ be in ${\mathbf \Pi}^0_{1+\xi}$ with $\xi>0$. There exists a sequence 
$B_n$, $n\in {\mathbb N}$, with $B_n\in {\mathbf \Pi}^0_{\gamma_n}$, 
for some $\gamma_n<\xi$, 
with $X\setminus A = \bigcup_n B_n$. We have 
\[
c_\xi(A)\setminus A = c_\xi(A)\cap \bigcup_n B_n \subseteq \bigcup_n \bigl(c_\xi(A) \cap c_{\gamma_n}(B_n)\bigr) \cup \bigl( B_n\setminus c_{\gamma_n}(B_n)\bigr).  
\] 
By what we proved above, the set $B_n\setminus c_{\gamma_n}(B_n)$ is slight for each $n$, so also $\tau_\alpha$-meager, by Lemma~\ref{L:onsl}(ii); 
thus, to prove the conclusion of the lemma, by Lemma~\ref{L:onsl}(iii), 
it suffices to show that, for each $n$, $c_\xi(A) \cap c_{\gamma_n}(B_n)$ is slight, if $\xi<\alpha$, and is $\tau_\alpha$-meager, if $\xi=\alpha$. 
Both these goals will be achieved 
if we prove that 
\[
{\rm int}_\alpha (c_\xi(A) \cap c_{\gamma_n}(B_n)) =\emptyset.
\]
This equality will follow from the observation at the beginning of the proof if we show that $A\cap c_{\gamma_n}(B_n)$ is slight and $A\cap V$ 
is not slight for any $\tau_\xi$-open set $V$ 
with $V\cap c_\xi(A)\not=\emptyset$. The second condition holds by the definition of $c_\xi(A)$. To see that $A\cap c_{\gamma_n}(B_n)$ is slight, note that 
\[
A\cap c_{\gamma_n}(B_n)\subseteq c_{\gamma_n}(B_n)\setminus B_n,
\]
and by our inductive assumption $c_{\gamma_n}(B_n)\setminus B_n$ is slight. 
\end{proof}

\begin{proof}[Proof of Lemma~\ref{L:stab}]
Let $A$ be ${\mathbf \Pi}^0_{1+\xi}$, $\xi\leq\alpha$, and let $B\subseteq A$ be $\alpha$-solid. By Lemmas~\ref{L:slal} and \ref{L:onsl}(ii), 
independently of whether 
$\xi<\alpha$ or $\xi=\alpha$, there exists a $\tau_\xi$-closed 
set $F$ such that $F\setminus A$ is $\tau_\alpha$-meager and 
$A\setminus F$ is covered by $\tau_\alpha$-open sets $U\subseteq X\setminus F$ with $A\cap U$ slight. 
Note that this last statement together with the assumption that $B$ is $\alpha$-solid
immediately imply that $B\setminus F$ is empty.  
Thus, we have $B\subseteq F$. Since $F$ is $\tau_\xi$-closed, it follows that ${\rm cl}_\xi(B)\subseteq F$, which gives 
\[
{\rm cl}_\xi(B)\setminus A\subseteq F\setminus A. 
\] 
Since $F\setminus A$ is $\tau_\alpha$-meager, we have that ${\rm cl}_\xi(B)\setminus A$ is $\tau_\alpha$-meager, as required. 
\end{proof}

\begin{proof}[Proof of Theorem~\ref{T:stab2}]
We continue with our convention \eqref{E:shclin}. 
First, we note that each non-empty $\tau$-open set $B$ is $\alpha$-solid. Indeed, let each $F_n$, $n\in {\mathbb N}$, be $\alpha$-tame. 
Assume that $\bigcup_n F_n$  
contains a non-empty relatively $\tau_\alpha$-open subset of $B$. So $\bigcup_nF_n$ contains a non-empty $\tau$-open set. 
Since, by Lemma~\ref{L:fsi}(i), each $F_n$ is a countable union of $\tau$-closed sets 
and $\tau$ is Baire, we have ${\rm int}_{\tau}(F_n)\not= \emptyset$ for some $n$. By Lemma~\ref{L:frth}(i), 
we have ${\rm int}_{\alpha}(F_n)\not= \emptyset$. Thus, $B$ is $\alpha$-solid.

Now we show that if $A\subseteq X$ is a $\tau$-neighborhood of $x$, 
then $A$ is $\tau_\alpha$-comeager in a $\tau_\alpha$-neighborhood of $x$. We can assume that $A$ is ${\mathbf \Pi}^0_{1+\xi}$ for some $\xi<\alpha$. 
Note that $B={\rm int}_{\tau}(A)$ is $\tau$-open and $x\in B$. Since, by what was proved above, 
$B$ is $\alpha$-solid, by Lemma~\ref{L:stab}, we have that ${\rm cl}_\xi(B)\setminus A$ is $\tau_\alpha$-meager. Put $F= {\rm cl}_\xi(B)$ and note that $F$ 
is $\tau_\xi$-closed. By assumption \eqref{E:intap2}, 
we get 
\[
{\rm int}_\alpha(F) = {\rm int}_{\tau}(F)\supseteq B\ni x.
\]  
Clearly we also have 
\[
{\rm int}_\alpha(F) \setminus A\subseteq F\setminus A, 
\]
and this last set is $\tau_\alpha$-meager. Thus, ${\rm int}_\alpha(F)$ is the desired $\tau_\alpha$-neighborhood of $x$. 
 
The above observation implies the conclusion of the theorem by Lemma~\ref{L:toe} 
applied to ${\rm id}_X\colon (X, \tau)\to (X, \tau_{\alpha})$. 
\end{proof}

Recall the notation $\alpha\oplus 1$ from \eqref{E:oplu}. 

\begin{corollary}\label{C:stom}
Let $\sigma\subseteq \tau$ be topologies, 
with $\tau$ being regular and Baire. For an ordinal $\alpha\leq\omega_1$, let 
$(\tau_\xi)_{\xi<\alpha\oplus 1}$ be a filtration from $\sigma$ to $\tau$, with $\tau_\xi$ completely metrizable for each $\xi<\alpha\oplus 1$. 

If $\tau$ has a neighborhood basis consisting of sets that are in 
$\bigcup_{\xi<\alpha}{\mathbf \Pi}^0_{1+\xi}$ with respect to $\sigma$, then $\tau=\bigvee_{\xi<\alpha}\tau_\xi$.
\end{corollary}

\begin{proof}  If $\alpha$ is a successor ordinal, then $\alpha\oplus 1 =\alpha+1$, and the conclusion is immediate from Theorem~\ref{T:stab2}. 

Assume now that $\alpha$ is limit. (The corollary is tautologically true for $\alpha=0$.) We then have $\alpha\oplus 1 =\alpha$. Let 
\[
\tau_{\alpha}=\bigvee_{\xi<\alpha}\tau_\xi.
\]
Note that $(\tau_\xi)_{\xi\leq\alpha}$ is a filtration from $\sigma$ to $\tau$. 
If we show that $\tau_{\alpha}$ is Baire, Theorem~\ref{T:stab2} will imply that $\tau=\tau_{\alpha}$ as required. 
Therefore, it suffices to check the following claim. 

\begin{claim}
Let $T$ be a set of completely metrizable topologies linearly ordered by inclusion. Then $\bigvee T$ is a Baire topology. 
\end{claim}

\noindent {\em Proof of Claim.} Let $F_i$, $i\in {\mathbb N}$, be a sequence of sets that are nowhere dense with respect to $\bigvee T$, and let $U$ be a non-empty 
set that is open with respect to $\bigvee T$. Since $T$ is linearly ordered by inclusion, we can assume that $U\in t$ for some $t\in T$. 

We inductively construct 
topologies $t_i\in T$, $i\in {\mathbb N}$, with a complete metric $d_i\leq 1$ inducing $t_i$. We also construct non-empty sets $U_i\in t_i$. All this is arranged so that 
$U_i\cap F_i=\emptyset$, $d_i{\rm -diam}(U_j)\leq \frac{1}{j+1}$, $t\subseteq t_i\subseteq t_j$, ${\rm cl}_{t_i}(U_j)\subseteq U_i\subseteq U$ for all natural numbers $i<j$. 
The construction of these objects is easy using the fact that $T$ is linearly ordered by inclusion. 

Consider now the topology $t_\infty=\bigvee_{i\in {\mathbb N}} t_i$, which is completely metrizable as witnessed by the metric $d_\infty = \sum_i 2^{-i}d_i$. Note that the 
sets ${\rm cl}_{t_\infty}(U_i)$ are non-empty, $t_\infty$-closed, decreasing, and their $d_\infty$-diameters tend to $0$. It follws that 
their intersection consists of precisely one point $x_\infty$. For each $i$ we have 
\[
x_\infty\in {\rm cl}_{t_\infty}(U_{i+1})\subseteq {\rm cl}_{t_i}(U_{i+1})\subseteq U_i. 
\]
Thus, $x_\infty\in U\setminus \bigcup_{i\in {\mathbb N}} F_i$. 

We just showed that the complement of $\bigcup_{i \in {\mathbb N}}F_i$ is dense with respect to $\bigvee T$, and the claim follows.  
\end{proof}

\section{Upper approximations of equivalence relations}\label{S:eqr}

Fix $(\tau_\xi)_{\xi<\rho}$, a transfinite sequence of topologies as in \eqref{E:cot}. 
Let $E$ be an equivalence relation on $X$. There exists a natural way of producing 
a transfinite sequence of upper approximations of $E$ using $(\tau_\xi)_{\xi<\rho}$. For each $\xi<\rho$ define the equivalence relation $E_\xi$ on 
$X$ by letting 
\[
xE_\xi y\;\hbox{ if and only if }\; {\rm cl}_{\tau_\xi}([x]_E)= {\rm cl}_{\tau_\xi}([y]_E). 
\]
Note that 
\begin{equation}\label{E:down}
E_0\supseteq E_1\supseteq\cdots \supseteq E_\xi\supseteq \cdots \supseteq E.
\end{equation}
The main question is when the transfinite sequence of equivalence relations in \eqref{E:down} stabilizes at $E$. Theorem~\ref{T:eqte} gives an answer.

Before we state it we need a definition. Let $(\tau_\xi)_{\xi<\rho}$ be a transfinite sequence of topologies with \eqref{E:cot}. 
Recall the definition of $\alpha$-solid for $\alpha<\rho$ from Section~\ref{S:stdes}. We call a set {\bf solid} if it is $\alpha$-solid for each $\alpha<\rho$. 
(For more on this notion, see Remark 2 following the statement of Theorem~\ref{T:eqte}.) 
Recall notation $\alpha\oplus 1$ from \eqref{E:oplu}. 

\begin{theorem}\label{T:eqte}
Let $(\tau_\xi)_{\xi<\alpha\oplus 1}$, $\alpha \leq\omega_1$, be a filtration from $\sigma$. 
Assume $\tau_\xi$ is completely metrizable for each $\xi<\alpha$. 
Let $E$ be an equivalence relation on $X$ whose equivalence classes are solid. 

If all equivalence classes of $E$ are in $\bigcup_{\xi<\alpha} {\bf \Pi}^0_{1+\xi}$ with respect to $\sigma$, 
then $E=\bigcap_{\xi<\alpha}E_\xi$. 
\end{theorem}


\begin{remark}
We keep the notation and assumptions as Theorem~\ref{T:eqte}. 

{\bf 1.} If $\alpha$ is a successor ordinal, say $\alpha=\beta+1$, then the conclusion of Theorem~\ref{T:eqte} reads: if all equivalence classes of $E$ are in 
${\bf \Pi}^0_{1+\beta}$ with respect to $\sigma$, 
then $E=E_\beta$. 

{\bf 2.} We point out here that being solid, under the assumptions of Theorem~\ref{T:eqte}, can be phrased so that it involves 
only sets that are $\tau_\xi$-closed for appropriate $\xi$ rather than the more complicated $\alpha$-tame sets.

If $\alpha$ is a successor ordinal, then $A$ being solid means $\alpha$-solid. So, under the assumption that $\tau_\xi$ is metrizable 
for each $\xi<\alpha$, by Lemma~\ref{L:onso}, $A$ being solid 
is equivalent to the following condition: for each countable family $\mathcal F$ with every $F\in {\mathcal F}$ being  
$\tau_\xi$-closed, where $\xi<\alpha$ depends on $F$,  and 
with $\bigcup {\mathcal F}$ containing a non-empty relatively 
$\tau_\alpha$-open subset of $A$, we have ${\rm int}_{\tau_\alpha}(F)\not= \emptyset$ for some $F\in {\mathcal F}$. 

In the same spirit, if $\alpha$ is limit, then $A$ being solid means $\alpha'$-solid for each $\alpha'<\alpha$. So, again, under the assumption that $\tau_\xi$ is metrizable 
for each $\xi<\alpha$, by Lemma~\ref{L:onso}, $A$ being solid 
is equivalent to the following condition: for each $\alpha'<\alpha$, for each countable family $\mathcal F$ with every $F\in {\mathcal F}$ being  
$\tau_\xi$-closed, where $\xi<\alpha'$ depends on $F$,  and 
with $\bigcup {\mathcal F}$ containing a non-empty relatively 
$\tau_{\alpha'}$-open subset of $A$, we have ${\rm int}_{\tau_{\alpha'}}(F)\not= \emptyset$ for some $F\in {\mathcal F}$. 
\end{remark}

We will need a refinement of a special case of Lemma~\ref{L:stab}. Our gain consists of getting ${\rm cl}_{\tau_\alpha}(A)\setminus A$ to be relatively $\tau_\alpha$-meager 
in ${\rm cl}_{\tau_\alpha}(A)$ rather than just $\tau_\alpha$-meager. In exchange, we have to make a stronger assumption 
that $A$ be $(\alpha+1)$-solid rather than $\alpha$-solid 
(see Lemma~\ref{L:frth}(ii)). Recall the definition of weak filtration from Section~\ref{S:stdes}. 
The proof of the lemma below is the place where we need to use weak filtrations 
instead of filtrations.

\begin{lemma}\label{L:stabre}
Let $\alpha<\omega_1$. 
Let $(\tau_\xi)_{\xi\leq\alpha+1}$ be a filtration from $\sigma$, with $\tau_\xi$ metrizable for each $\xi\leq \alpha$. 
If $A$ is $(\alpha+1)$-solid and ${\bf \Pi}^0_{1+\alpha}$ with respect to $\sigma$, 
then ${\rm cl}_{\tau_{\alpha}}(A)\setminus A$ is relatively $\tau_\alpha$-meager in ${\rm cl}_{\tau_\alpha}(A)$. 
\end{lemma}

\begin{proof}
The conclusion will follow from Lemma~\ref{L:stab}. Put $X' = {\rm cl}_{\tau_\alpha}(A)$, and let $\tau_\xi'$ be $\tau_\xi$ restricted to $X'$. 
Note that $(\tau_\xi')_{\xi\leq\alpha}$ is a transfinite sequence of topologies on $X'$ fulfilling \eqref{E:cot} 
with $\tau_\xi'$ metrizable for $\xi\leq\alpha$.

First, we check that $A$ being $(\alpha+1)$-solid with respect to $(\tau_\xi)_{\xi\leq \alpha+1}$ implies that it is $\alpha$-solid with respect to $(\tau'_\xi)_{\xi\leq\alpha}$. 
By Lemma~\ref{L:onso},
it suffices to check that for every sequence $(F_n')$ such that $F_n'$ is $\tau_{\xi_n}'$-closed, for some $\xi_n<\alpha$, and $\bigcup_n F_n'$ contains
a non-empty relatively $\tau_\alpha'$-open subset of $A$, there is $n$ such that ${\rm int}_{\tau_\alpha'}(F_n')\not=\emptyset$. Let $F_n$ be $\tau_{\xi_n}$-closed with 
$F_n'=F_n\cap X'$. Our assumption on $(F_n')$ implies that 
$\bigcup_n \left( F_n\cap X'\right)$ contains a non-empty relatively $\tau_\alpha$-open subset of $A$ since $A$ 
is a subset of $X'$. Now consider the countable family $\{ F_n\cap X'\mid n\in {\mathbb N}\}$. 
Since $X'$ is $\tau_\alpha$-closed, the sets $F_n\cap X'$ are $\tau_\alpha$-closed. Since 
$A$ is $(\alpha+1)$-solid with respect to $(\tau_\xi)_{\xi\leq\alpha+1}$, there is $n$ such that 
\begin{equation}\label{E:lat}
{\rm int}_{\tau_{\alpha+1}}(F_n\cap X')\not=\emptyset.
\end{equation}
Since $(\tau_\xi)_{\xi\leq\alpha+1}$ is a filtration, equation \eqref{E:lat} gives 
${\rm int}_{\tau_{\alpha}}(F_n)\cap X'\not=\emptyset$, that is, 
\[
{\rm int}_{\tau_{\alpha}'}(F_n\cap X')\not=\emptyset, 
\]
as required. 

Thus, to reach the desired conclusion by using Lemma~\ref{L:stab} (applied to $A=B$ and $X'$ with $(\tau_\xi')_{\xi\leq\alpha}$), 
it suffices to check condition \eqref{E:intap} for $(\tau_\xi')_{\xi\leq\alpha}$, which we now do. 
Let $\xi<\beta<\alpha$, and let $F$ 
be $\tau_\xi$-closed. We need to check that ${\rm int}_{\tau'_\beta}(F\cap X')$ is $\tau_\alpha'$-dense in ${\rm int}_{\tau'_\alpha}(F\cap X')$. 
Let 
\[
N = {\rm int}_{\tau_{\alpha+1}}(X').
\]
Since $A$ is solid, $N$ is $\tau_\alpha$-dense in $X'$. It will therefore suffice to check that 
\begin{equation}\label{E:nee}
{\rm int}_{\tau'_\alpha}(F\cap X')\cap N \subseteq {\rm int}_{\tau'_\beta}(F\cap X'). 
\end{equation}
Observe that 
\begin{equation}\label{E:pri}
{\rm int}_{\tau'_\alpha}(F\cap X')\cap N \subseteq {\rm int}_{\tau'_{\alpha+1}}(F\cap X')\cap N, 
\end{equation}
and, further, since $N$ is $\tau_{\alpha+1}$-open and included in $X'$, we have 
\begin{equation}\label{E:plus}
{\rm int}_{\tau'_{\alpha+1}}(F\cap X')\cap N \subseteq {\rm int}_{\tau_{\alpha +1}}(F\cap N) = {\rm int}_{\tau_{\alpha +1}}(F) \cap N.
\end{equation}
Note that we use $\tau_{\alpha+1}$-openness of $N$ in our verification of the inclusion and the equality in \eqref{E:plus}. 
On the other hand, we have  
\begin{equation}\label{E:ir}
{\rm int}_{\tau_\beta}(F)\cap X'\subseteq {\rm int}_{\tau'_\beta}(F\cap X'). 
\end{equation}
By combining \eqref{E:pri}, \eqref{E:plus}, and \eqref{E:ir}, we see that to prove \eqref{E:nee}, 
it is enough to show 
\[
{\rm int}_{\tau_{\alpha +1}}(F)\cap N\subseteq {\rm int}_{\tau_\beta}(F)\cap X'. 
\]
Since $N\subseteq X'$, this inclusion immediately follows from ${\rm int}_{\tau_{\alpha +1}}(F)\subseteq {\rm int}_{\tau_\beta}(F)$. 
\end{proof}

\begin{proof}[Proof of Theorem~\ref{T:eqte}]
Let $x,y\in X$ be such that $[x]_E$ and $[y]_E$ are ${\bf \Pi}^0_{1+\xi}$ for some $\xi<\alpha$. If 
$xE_\xi y$, then ${\rm cl}_\xi([x]_E) = {\rm cl}_\xi([y]_E)$. 
Note that $\xi+1<\alpha\oplus 1$. Using this assumption, metrizability of $\tau_\gamma$ for 
$\gamma<\xi$, and $[x]_E$ and $[y]_E$ being $(\xi+1)$-solid, by Lemma~\ref{L:stabre}, we see that 
$[x]_E$ and $[y]_E$ are both $\tau_\xi$-comeager in
${\rm cl}_\xi([x]) = {\rm cl}_\xi([y])$. This last set is $\tau_\xi$-closed and $\tau_\xi$ is completely metrizable, so $[x]_E$ and $[y]_E$ intersect; thus, $xEy$. 

It follows from the argument above that if each $E$ equivalence class is in in the family $\bigcup_{\xi<\alpha}{\bf \Pi}^0_{1+\xi}$ with respect to $\sigma$, then 
$\bigcap_{\xi<\alpha}E_\xi \subseteq E$, so $E=\bigcap_{\xi<\alpha}E_\xi$, as required. 
\end{proof}

The following corollary is a consequence of Theorem~\ref{T:eqte}, in which the assumption on equivalence classes is phrased in terms of $\tau$. We 
emphasize that no separability assumptions are needed.

\begin{corollary}\label{C:ceq}
Let $\sigma\subseteq \tau$ be topologies, 
with $\tau$ being Baire. Let $\alpha\leq\omega_1$, and let 
$(\tau_\xi)_{\xi<\alpha}$ be a filtration from $\sigma$ to $\tau$, with $\tau_\xi$ completely metrizable for each $\xi<\alpha$. 
Assume $E$ is an equivalence relation whose equivalence classes are $\tau$-open.

If all $E$ equivalence classes are in $\bigcup_{\xi<\alpha} {\bf \Pi}^0_{1+\xi}$ with respect to $\sigma$, then $E=\bigcap_{\xi<\alpha}E_\xi$. 
\end{corollary}

\begin{remark} 
{\bf 1.} Each $E$ equivalence class being $\tau$-open, as in the corollary above, is equivalent to saying that $E$ is a $(\tau\times\tau)$-open subset of $X\times X$. 

{\bf 2.} As in Theorem~\ref{T:eqte}, in Corollary~\ref{C:ceq}, if $\alpha<\omega_1$ is a successor, 
say $\alpha=\beta+1$, then the conclusion reads: if all equivalence classes of $E$ are in 
${\bf \Pi}^0_{1+\beta}$ with respect to $\sigma$, 
then $E=E_\beta$. 
\end{remark}

\begin{proof}[Proof of Corollary~\ref{C:ceq}] Extend the given filtration to a filtration $(\tau_\xi)_{\xi< \alpha+1}$ by setting $\tau_\alpha=\tau$. 
It follows from Theorem~\ref{T:eqte}, that it suffices to check that $E$ equivalence classes are solid. Since $(\tau_\xi)_{\xi< \alpha+1}$ is 
a filtration from $\sigma$ to $\tau$, by Lemma~\ref{L:frth}(ii), it suffices to check that $E$ equivalence classes are $\alpha$-solid. This is immediate, 
by Lemma~\ref{L:onsl}(ii), 
from $\tau$ being Baire and each $E$-class being $\tau$-open. 
\end{proof}

\smallskip

\noindent {\bf Acknowledgments.} I would like to thank Assaf Shani for pointing out paper \cite{Dr} to me.

\end{document}